\documentclass[10pt,reqno]{amsart}
\usepackage{amssymb}
\usepackage[all]{xy}

\theoremstyle{plain}
\newtheorem{prop}{Proposition}
\newtheorem{theo}[prop]{Theorem}

\newtheorem{lemm}[prop]{Lemma}
\theoremstyle{remark}
\newtheorem{rema}[prop]{Remark}
\theoremstyle{definition}
\newtheorem{defi}[prop]{Definition}
\newtheorem{assu}[prop]{Assumption}

\numberwithin{equation}{section}

\newcommand{\C}{{\mathbb C}}
\newcommand{\PP}{{\mathbb P}}
\newcommand{\Q}{{\mathbb Q}}

\newcommand{\N}{{\mathbb N}}
\newcommand{\R}{{\mathbb R}}
\newcommand{\Z}{{\mathbb Z}}

\newcommand{\eqto}{\stackrel{\lower1.5pt\hbox{$\scriptstyle\sim\,$}}\to}
\DeclareMathOperator{\Gal}{Gal}

\DeclareMathOperator{\Pic}{Pic}
\DeclareMathOperator{\Spec}{Spec}
\DeclareMathOperator{\Proj}{Proj}

\DeclareMathOperator{\Hom}{Hom}
\DeclareMathOperator{\End}{End}
\DeclareMathOperator{\Br}{Br}

\begin{document}
\title[Picard and Brauer groups]
{Effective computation of Picard groups and Brauer-Manin obstructions
of degree two $K3$ surfaces over number fields}
\author{Brendan Hassett}
\address{
  Department of Mathematics,
  Rice University,
  Houston, TX 77005, USA
}
\email{hassett@rice.edu}
\author{Andrew Kresch}
\address{
  Institut f\"ur Mathematik,
  Universit\"at Z\"urich,
  Winterthurerstrasse 190,
  CH-8057 Z\"urich, Switzerland
}
\email{andrew.kresch@math.uzh.ch}
\author{Yuri Tschinkel}
\address{
  Courant Institute,
  251 Mercer Street,
  New York, NY 10012, USA
}
\email{tschinkel@cims.nyu.edu}

\begin{abstract}
Using the Kuga-Satake correspondence
we provide an effective algorithm for the computation of the
Picard and Brauer groups of $K3$ surfaces of degree $2$ 
over number fields.
\end{abstract}
\maketitle

\section{Introduction}
\label{sec:introduction}
Let $X$ be a
smooth projective variety over a number field $k$ and
$\Br(X)$ its Brauer group.
The quotient $\Br(X)/\Br(k)$
plays an important role in the study of arithmetic properties of $X$.
Its effective computation
is possible in certain cases, for example:
when $X$ is a geometrically rational surface (see, e.g., \cite{kst}),
a Fano variety of dimension at most 3 (\cite{kteff}),
or a diagonal $K3$ surface over $\Q$ (\cite{ISZ}, \cite{kt}, \cite{swinnertondyer}).

Skorobogatov and Zarhin proved the finiteness of
$\Br(X)/\Br(k)$ when $X$ is a $K3$ surface \cite{skorobogatovzarhin}.
Letting $X_{\bar k}$ denote $X\times_{\Spec k}{\Spec \bar k}$,
where $\bar k$ is an algebraic closure of $k$,
there is the natural map
\begin{equation}
\label{XtoXbar}
\Br(X)\to\Br(X_{\bar k}).
\end{equation}
Its kernel is known as algebraic part of the
Brauer group.
This is a finite group, and may be identified with the Galois cohomology
\begin{equation}
\label{h1}
H^1(\Gal(\bar k/k),\Pic(X_{\bar k})).
\end{equation}
Knowledge of $\Pic(X_{\bar k})$ is essential to its computation.
A first goal of this paper is the effective computation of
$\Pic(X_{\bar k})$ when $X$ is a $K3$ surface of degree 2 over a number field.
Special cases and examples have been treated previously by, e.g.,
van Luijk \cite{vl}, while a more general treatment, that is however
conditional on the Hodge conjecture, appears in \cite{charles}.

The image of \eqref{XtoXbar} is contained in the invariant
subgroup
\[\Br(X_{\bar k})^{\Gal(\bar k/k)}.\]
The finiteness of this invariant subgroup is
one of the main results of \cite{skorobogatovzarhin}, yet the proof does
not yield an effective bound.
In this paper we give an \emph{effective} bound
for the order of the group $\Br(X)/\Br(k)$.
Combined with the results in \cite{kt}, this permits the effective
computation of the subset
\[
X(\mathbb{A}_k)^{\Br(X)}\subseteq X(\mathbb{A}_k)
\]
of Brauer-Manin unobstructed adelic points of $X$.
Examples of computations of Brauer-Manin obstructions on $K3$ surfaces can
be found in
\cite{bright}, \cite{hvv}, \cite{ieronymou}, \cite{ssd}, \cite{wittenberg}.
The results here,
combined with results in \cite{ctskoro},
imply as well an effective
bound for the order of $\Br(X_{\bar k})^{\Gal(\bar k/k)}$.

The finiteness results of \cite{skorobogatovzarhin} are based on
the Kuga-Satake construction, which associates an abelian variety of
dimension $2^{19}$ to a given $K3$ surface
and relates their cohomology,
together with the Tate conjecture for abelian varieties,
proved by Faltings \cite{faltings}.
The Kuga-Satake correspondence is \emph{conjectured} to be given by an
algebraic correspondence,
but is proved only in some special cases, e.g.,
\cite{vangeemen}, \cite{inose}, \cite{shiodainose}.
Assuming this (in some effective form), one could apply
effective versions of Faltings' results, obtained by
Masser and W\"ustholz \cite{mw1}, \cite{mw2}, \cite{mw3}, \cite{mw4}, \cite{mw5}
(see also \cite{bost}).
Lacking this, we treat the transcendental construction directly, showing
that computations to bounded precision can replace an algebraic
correspondence, in practice.

The construction proceeds in several steps.
First of all, rigidity allows us to construct the Kuga-Satake morphism between
moduli spaces (of $K3$ surfaces with polarization and level structure on
one side and polarized abelian varieties with level structure on the other)
algebraically over a number field, at least up to an explicit finite
list of possibilities.
This allows us to identify an abelian variety corresponding to a $K3$ surface,
and its field of definition.
For the computation of the induced map on homology we work with
integer coefficients and simplicial complexes, and computations up to bounded
precision suffice to determine all the necessary maps.
Then we follow the proof in Section 4 of \cite{skorobogatovzarhin}
and obtain the following result.

\begin{theo}
\label{mainresult}
Let $k$ be a number field and $X$ a $K3$ surface of degree $2$ over $k$,
given by an explicit equation.
Then there is an effective bound
on the order of $\Br(X)/\Br(k)$.
\end{theo}

In fact, we provide an effective bound on $\Br(X)/\Br(k)$
when $X$ has an ample line bundle of arbitrary degree $2d$
provided that there is an effective construction of the moduli space of
primitively quasi-polarized $K3$ surfaces of degree $2d$
(see Definition \ref{def:pola}).
For $d=1$ this is known, via effective geometric invariant
theory (see Remark \ref{whenknowDmodGamma}).

The first step of the proof is, as mentioned above, the
effective computation of the Galois module $\Pic(X_{\bar k})$.
Since this is finitely generated and torsion-free,
this permits the effective computation of the Galois cohomology
group \eqref{h1}.
Hence the proof of Theorem \ref{mainresult} is
reduced to effectively bounding the image of \eqref{XtoXbar}.

\medskip
\noindent\textbf{Acknowledgements.}
The first author was supported by NSF grants 0901645 and 0968349.
The second author was supported by the SNF.
The third author was supported by NSF grants 0739380, 0901777, and 0968349.
The authors benefited from helpful discussions with
N. Katz and G. W\"ustholz.

\section{Effective algebraic geometry}
\label{sec:effalggeo}
We work over an algebraic number field $k$ and denote by $\bar k$ its
algebraic closure.
The term \emph{variety} refers to geometrically integral separated
scheme of finite type over $k$.
We say that a quasiprojective variety or scheme $X$
is given by explicit equations
if homogeneous equations are supplied defining a scheme in a projective space
$\PP^M$ for some $M$
and a closed subscheme whose complement is $X$.
By convention $\mathcal{O}_X(1)$ will denote the restriction of
$\mathcal{O}_{\PP^M}(1)$ to $X$.
The base-change of $X$ to a field extension $k'$ of $k$ will be denoted
$X_{k'}$.

\begin{lemm}
\label{finitemorlem}
Let $X$ be a quasiprojective scheme, given by explicit equations.
Let $f:Y\to X$ be a finite morphism, given by explicit equations
on affine charts.
Then we may effectively determine integers $n$ and $N$ and an embedding
$Y\to X\times \PP^N$, such that $f$ is the composite of projection to $X$
with the embedding and the pullback of $\mathcal{O}_{\PP^N}(1)$ to $X\times \PP^N$
restricts to $f^*\mathcal{O}_X(n)$.
In particular, $f^*\mathcal{O}_X(n+1)$ is very ample on $Y$ and
we may obtain explicit equations for $Y$ as a quasiprojective scheme.
\end{lemm}

\begin{proof}
The morphism $f$ may be presented on affine patches by
finitely many new indeterminates adjoined 
to the coordinate rings of the patches (with additional
relations).
We may determine, effectively, an integer $n$ such that
each extends to a section of $\mathcal{O}_X(n)$.
Then for suitable $N$ we have an embedding
$Y\to X\times \PP^N$ satisfying the desired conditions.
\end{proof}

We collect effectivity results that we will be using freely.

\begin{itemize}
\item \emph{Effective normalization in a finite function field extension}:
Given a quasiprojective variety $X$
over $k$, presented by means of explicit equations, and another algebraic
variety $V$ with generically finite morphism $V\to X$ given by
explicit equations on affine charts,
to compute effectively
the normalization $Y$ of $X$ in $k(V)$, with finite morphism $Y\to X$.
See \cite{mesnager} and references therein.
\item \emph{A form of effective resolution of singularities}:
Given a nonsingular quasiprojective variety $X^\circ$ over $k$,
to produce a nonsingular projective variety $X$ and
open immersion $X^\circ\to X$, such that
$X\smallsetminus X^\circ$ is a simple normal crossings divisor.
This follows by standard formulations of Hironaka resolution theorems,
for which effective versions are available; see, e.g., \cite{bgmw}.
\item \emph{Effective invariant theory for actions of projective varieties}:
Given a projective variety $X$
and a linearized action of a reductive algebraic group $G$ on $X$ for
$L=\mathcal{O}_X(1)$, to
compute effectively the subsets $X^{ss}(L)$ and $X^s(L)$ of
geometric invariant theory,
the projective variety $X//G$, open subset $U\subset X//G$
corresponding to $X^s(L)$, and quotient morphisms
$X^{ss}(L)\to X//G$ and $X^s(L)\to U$.
This is standard, using effective computation of invariants
$k[V]^G$ for a finitely generated $G$-module $V$
\cite{derksen}, \cite{kempf}, \cite{popov}.
\end{itemize}

\begin{lemm}
\label{hilbschemeargument}
Let $X$ be a quasiprojective normal variety over $k$,
given by explicit equations, and let $U\subset X$ be a nonempty subvariety.
Given $d\in\N$, there is an effective procedure to produce a finite
extension $k'$ of $k$ and
a finite collection of normal quasiprojective varieties
$Y^{(1)}$, $\ldots$, $Y^{(m)}$ over $k'$, with finite morphisms
$f_i:Y^{(i)}\to X_{k'}$,
such that for each $i$ the restriction of $f_i$ to $f_i^{-1}(U_{k'})$
is \'etale of degree $d$,
and such that the $Y^{(i)}_{\bar k}\to X_{\bar k}$ ($i=1$, $\ldots$, $m$)
are up to isomorphism all the degree $d$ coverings by
normal quasiprojective varieties over $\bar k$ which are
\'etale on the pre-image of $U_{\bar k}$.
\end{lemm}

\begin{proof}
Let $N=\dim X$.
Shrinking $U$ if necessary we may suppose that there is morphism
$U\to \PP^{N-1}$, given by a suitable linear projection, such that
the generic fiber is smooth and one-dimensional, i.e., the restriction
to the generic point $\eta=\Spec(k(x_1,\ldots,x_{N-1}))$ is a
nonsingular quasi-projective curve $C_{\eta}$.
We may effectively compute a nonsingular projective compactification
$\overline{C}_{\eta}$.
The genus $g=g(\overline{C}_{\eta})$ and degree
$e=\deg(\overline{C}_{\eta}\smallsetminus C_{\eta})$, together with $d$,
determine by the Riemann-Hurwitz formula an upper bound $g_{\mathrm{max}}$
on the genus of $Y\times_X \eta$.
Using effective Hilbert scheme techniques we construct parameter spaces
containing every isomorphism class of genus $g'\le g_{\mathrm{max}}$ curve
equipped with a morphism of degree $d$ to $\overline{C}_{\eta}$.
That the ramification divisor in contained in the
scheme-theoretic pre-image of $\overline{C}_{\eta}\smallsetminus C_{\eta}$
is a closed condition and one that may be implemented effectively
(cf.\ \cite[Corollary 3.14]{mochizuki}) and determines a finite field extension
$k'$ of $k$ and a finite set of
candidates for
$Y\times_X \eta$, over $k'$.
Applying effective normalization to each of these candidates and
eliminating those which are not \'etale over $U_{k'}$,
we obtain the $Y^{(i)}$.
\end{proof}

\section{Baily-Borel compactifications}
\label{sec:bb}
Let $\mathbb{D}=G/K$ be a bounded symmetric domain and
$\Gamma$ an arithmetic subgroup of $G$.
It is known that $\Gamma\backslash \mathbb{D}$ admits a canonical compactification
$(\Gamma\backslash \mathbb{D})^*$,
the Baily-Borel compactification, which is a normal projective variety
\cite{bb};
however, the construction of this variety does not supply algebraic equations.
Let $\widetilde{\Gamma}<\Gamma$ be a finite-index subgroup, and assume that
$\widetilde{\Gamma}$ is neat.
(Recall that an arithmetic subgroup is called neat if, for every element,
the subgroup of $\C^*$ generated by its eigenvalues is torsion-free.)
In this section, we show that if we know
$\Gamma\backslash \mathbb{D}$ as a quasiprojective variety (with explicit
equations over some number field),
we can effectively construct $\widetilde{\Gamma}\backslash \mathbb{D}$ as
a quasiprojective variety (as one of finitely many candidates),
together with its Baily-Borel compactification.

In the following, we let $k$ denote a number field.

\begin{lemm}
\label{neatlem}
Let $G$ be a linear algebraic group over $k$ and
$\Gamma$ an arithmetic subgroup.
There is an effective procedure to construct a neat subgroup of
finite index in $\Gamma$.
\end{lemm}

\begin{proof}
It suffices to establish the result for a discrete subgroup of
$GL_n(\Z)$.
Fix a prime $\ell$.
There is a finite extension $\mathbf{k}$ of $\Q_\ell$
over which every polynomial of degree $n$ with coefficients in $\Q_\ell$
factors completely
(see \cite{krasner},
effectively computed in \cite{pr}).
The structure of $\mathbf{k}^*$ is known as a direct sum of
$\Z$, a finite group, and
$\Z_\ell^N$ for some $N$.
Then there is an
$\varepsilon>0$ such that ball of radius $\varepsilon$ around $1$
is contained in the free part, so an $n\times n$
integer matrix sufficiently
close $\ell$-adically to the identity matrix has its eigenvalues not
more than $\varepsilon$ away from $1$.
Hence they generate a
torsion free subgroup of $\mathbf{k}^*$, and also of $\C^*$.
\end{proof}

We fix an embedding $k\hookrightarrow \C$.

\begin{prop}
\label{gammacover}
Let $\mathbb{D}=G/K$ be a bounded symmetric domain,
$\Gamma$ an arithmetic subgroup of $G$,
and $\widetilde{\Gamma}$ a finite-index subgroup of $\Gamma$ which is neat.
Let $X^\circ$ be a quasiprojective variety over $k$,
given by explicit equations, and let $U\subset X^\circ$ be a nonempty
open subscheme, also explicitly given.
Suppose that there exists an isomorphism
$X^\circ_{\C}\eqto \Gamma\backslash \mathbb{D}$ such that
the map $\mathbb{D}\to \Gamma\backslash \mathbb{D}$ is unramified over
the image of $U_\C$.
Then there is an effective procedure to produce a finite extension $k'$
of $k$ with compatible embedding in $\C$ and a finite collection
of nonsingular quasiprojective varieties $Y^{(1)}$, $\ldots$, $Y^{(m)}$
defined over $k'$ with morphisms
$f_i:Y^{(i)}\to X^\circ_{k'}$, such that, for some $i$,
setting $\widetilde{X}^\circ:=Y^{(i)}$ there exists
an isomorphism
$\widetilde{X}^\circ_{\C}\eqto \widetilde{\Gamma}\backslash \mathbb{D}$
fitting into a commutative diagram
\[
\xymatrix{
\widetilde{X}^\circ_{\C} \ar[r]^\sim \ar[d] &
\widetilde{\Gamma}\backslash \mathbb{D} \ar[d] \\
X^\circ_{\C} \ar[r]^\sim & \Gamma\backslash \mathbb{D}
}
\]
\end{prop}

\begin{proof}
Since $\widetilde{\Gamma}\backslash \mathbb{D}\to \Gamma\backslash \mathbb{D}$
ramifies over the same set of points as
$\mathbb{D}\to \Gamma\backslash \mathbb{D}$,
this follows directly from Lemma \ref{hilbschemeargument}.
\end{proof}

\begin{prop}
\label{neatbb}
Let $\mathbb{D}=G/K$ be a bounded symmetric domain,
$\Gamma$ a neat arithmetic subgroup of $G$, and
$X^\circ$ a quasiprojective variety over $k$ given by explicit equations
such that $X^\circ_{\C}$ is isomorphic to $\Gamma\backslash \mathbb{D}$.
Assume that $PGL_2$ is not a quotient of $G$.
Then there is an effective procedure to construct
a projective variety $X$ over $k$, together with open immersion
$X^\circ\to X$, such that
$X_{\C}$ is isomorphic to the
Baily-Borel compactification $(\Gamma\backslash \mathbb{D})^*$.
\end{prop}

The first ingredient in the proof of
Proposition \ref{neatbb}
is a result of Alexeev \cite[\S 3]{alexeev}, building on
earlier work of Mumford \cite{mumford}:
\begin{theo}
\label{alexeev}
Let $\mathbb{D}$ be a bounded Hermitian symmetric domain and $\Gamma$ a neat
arithmetic subgroup acting on $\mathbb{D}$.  Let $X^{\circ}=\Gamma\backslash \mathbb{D}$
with Baily-Borel compactification $X$ and boundary $\Delta$.
Then $(X,\Delta)$ is log canonical, with the automorphic factor
coinciding with the log canonical divisor $K_{X}+\Delta$.
\end{theo}

We will also use a result of Fujino \cite{fujino}:
\begin{theo}
\label{fujino2}
Let $(X,\Delta)$ be a projective log canonical pair and $M$ a line bundle on $X$.
Assume that $M\equiv K_X+\Delta+N$, where $N$ is an ample $\Q$-divisor on $X$.
Let $x_1,x_2 \in X$ be closed points and assume there there are positive numbers $c(k)$
with the following properties:
\begin{enumerate}
\item{If $Z\subset X$ is an irreducible (positive-dimensional) subvariety which contains
$x_1$ or $x_2$ then
$$(N^{\dim(Z)}\cdot Z) > c(\dim(Z))^{\dim(Z)}.$$ }
\item{The numbers $c(k)$ satisfy the inequality
$$\sum_{k=1}^{\dim(X)} \sqrt[k]{2} \frac{k}{c(k)} \le 1.$$
}
\end{enumerate}
Then the global sections of $M$ separate $x_1$ and $x_2$.
\end{theo}

\begin{proof}[Proof of Proposition \ref{neatbb}]
The hypotheses guarantee that the complement of $\Gamma\backslash \mathbb{D}$
in the Baily-Borel compactification has codimension $\ge 2$.
If we define
\[X=\Proj\Big(\bigoplus_{d\ge 0} H^0(X^\circ, dK_{X^\circ})\Big).\]
then $X$ satisfies the conditions of the proposition.
It remains to show that we can construct $X$ effectively.

Effective resolution of singularities as in Section \ref{sec:effalggeo} allows us
to construct a nonsingular compactification $\widetilde{X}$ of $X^\circ$,
projective, such that $\widetilde{X}\smallsetminus X^\circ$ is a
simple normal crossings divisor
$D_1\cup\cdots\cup D_m$.

Now the Borel extension property \cite[Thm.\ A]{borel} implies that
the inclusion $X^\circ\to X$ extends to a
birational morphism $\pi:\widetilde{X}\to X$.
By Theorem \ref{alexeev}, $X$ has at worst log canonical singularities.
Hence there are integers $c_i\le 1$ such that
\[\pi^*K_X=K_{\widetilde{X}}+\sum c_iD_i.\]
By the chain of inclusions
\[H^0(\widetilde{X},d(\pi^*K_X))
\subseteq
H^0(\widetilde{X},d(K_{\widetilde{X}}+\sum D_i))
\subseteq
H^0(X^\circ,dK_{X^\circ})\]
we deduce that
\begin{equation}
\label{KplussumD}
H^0(X,dK_X)=H^0(\widetilde{X},d(K_{\widetilde{X}}+\sum D_i)).
\end{equation}

Theorem \ref{fujino2} supplies a universal constant $n$ depending only on
the dimension of $X$, such that for any $d\ge n$, the linear system
$|dK_X|$ separates points on $X$.
The image $X'$ of $|dK_X|$ may be effectively computed using
\eqref{KplussumD}.
The normalization of $X'$, which may also be computed effectively,
is then isomorphic to $X$.
\end{proof}

\begin{rema}
\label{whenknowDmodGamma}
There are examples in the literature in which $X^\circ$ as in
Proposition \ref{gammacover} has been constructed.
\begin{itemize}
\item Abelian varieties
\begin{itemize}
\item Polarized, with
symmetric theta structure \cite{mumfordequationsiii}, \cite{mumfordtataiii}.
\item Polarized, with level $n$ structure ($n\ge 3$):
a construction based on Hilbert scheme and geometric invariant
theory, presented in \cite[\S7.3]{git}, can be carried out effectively
by the techniques mentioned in Section \ref{sec:effalggeo}.
\end{itemize}
\item $K3$ surfaces
\begin{itemize}
\item
A six-dimensional ball quotient as a moduli space of
$K3$ surfaces which are cyclic degree $4$ covers of $\PP^2$ ramified
along a quartic \cite{artebani}, \cite{kondocrelle}.
\item
A nine-dimensional ball quotient coming from $K3$ surfaces which are
cyclic triple covers of $\PP^1\times \PP^1$ with branch curve
of bidegree $(3,3)$ \cite{kondo}.
\item
$K3$ surfaces of degree $2$, described by Horikawa \cite{horikawa} and
Shah \cite{shahdegree2}; see also \cite{looijengavancouver}.
		
\end{itemize}
\item
A four-, resp.\ ten-dimensional ball quotient arising
from the moduli space of cubic surfaces \cite{ACT1},
resp.\ threefolds \cite{ACT2}.
\end{itemize}
Mumford's construction in the abelian variety setting yields
explicit equations for the moduli space together with a universal family.
In each of the other examples, explicit GIT constructions are given,
and from these we may obtain explicit equations as mentioned in
Section \ref{sec:effalggeo}.
This allows us,
e.g., to compute the point in $X^\circ$ corresponding to
a given $K3$ surface of degree $2$ presented as a double cover of
the plane branched along an explicitly given sextic curve.
\end{rema}

\begin{rema}
$K3$ surfaces of degree $4$ are analyzed in
\cite{looijengaamsterdam} and \cite{shahdegree4},
and degrees up to $8$ in \cite{looijengaduke}, via the GIT of
quartic surfaces, respectively complete intersections.
The analysis yields a factorization of the rational map from the
Baily-Borel compactification to the GIT quotient.
It would be interesting to use this to
give an effective construction of $X^\circ$ for these degrees
generalizing the one for degree $2$, which is based on
an explicit weighted Kirwan blowup
of the GIT quotient of plane sextics \cite{kirwanlee}.
\end{rema}

\begin{rema}
\label{igusa}
An effective construction of a Baily-Borel compactification is
tantamount to effectively bounding degrees of generators of the corresponding
ring of automorphic functions.
The technique in \cite{bb}
for proving the \emph{existence} of projective compactifications is
\emph{not effective} as it relies on a compactness argument.
In some examples these rings have been computed explicitly:
Igusa \cite{Ig62} shows that the ring of Siegel modular forms
for principally polarized abelian surfaces are generated
by Eisenstein series of weights $4,6,10$, and $12$.
The case of threefolds is explored by Tsuyumine \cite{tsu}, who shows
that $34$ modular forms, of weights ranging from $4$ to $48$, suffice.
Not all of these may be expressed in terms of Eisenstein series.
The case of fourfolds is addressed in Freitag-Oura \cite{FO}, who
introduce some specific relations and dimension formula.
Additional work in this direction was done by Oura-Poor-Yuen
\cite{opy}.
\end{rema}

\begin{prop}
\label{rigidity}
Let $k$ be a number field with a given embedding in $\C$.
Let $X$ and $X'$ be projective varieties over $k$
satisfying $X_{\C}\cong (\Gamma\backslash \mathbb{D})^*$ and
$X'_{\C}\cong (\Gamma'\backslash \mathbb{D}')^*$,
and suppose that $\Gamma'$ is neat.
Fix an integer $d$.
Then there is an effective procedure to produce a finite extension $k'$ of
$k$ with compatible embedding in $\C$ and morphisms
$f_1$, $\ldots$, $f_m:X_{k'}\to X'_{k'}$ such that
$(f_1)_{\C}$, $\ldots$, $(f_m)_{\C}$ are all the morphisms
$X_{\C}\to X'_{\C}$ under which the pullback of $K_{X'_{\C}}$ is
isomorphic to $dK_{X_{\C}}$.
\end{prop}

\begin{proof}
The Hilbert scheme representing morphisms of the given degree from
$X$ to $X'$ may be constructed effectively and by
rigidity (\cite{mok}) has dimension zero.
\end{proof}

\section{Abelian varieties}
\label{sec:av}

Let $A$ be an abelian variety over a number field $k$.

\begin{lemm}
\label{lem:semistable}
There is an effective way to produce a finite extension $k'$ of $k$ such that
$A$ acquires semistable reduction after base change to $k'$.
\end{lemm}

\begin{proof}
This is done by Proposition 4.7 of \cite{SGA7IX}.
\end{proof}

We recall two notions of heights of abelian varieties.
The Faltings height is computed using a semistable model.
Let $K$ be a finite extension of $k$ and
$\mathcal{A}$ a semi-stable model over $\mathfrak{o}_K$.
Then the Faltings height is the arithmetic degree of a particular
metrized canonical sheaf on $\mathcal{A}$
\[h_F(A)=\frac{1}{[K:\Q]}\widehat{\deg} \overline{\omega}_{\mathcal{A}/\Spec(\mathfrak{o}_K)}.\]
This is idependent of the choices of $K$ and $\mathcal{A}$.
For details see, e.g., \cite{bost} \S 2.1.3.

Alternatively, the theta height is defined purely algebraically, in terms
of a principal polarization.
In the following two effectivity results,
the complexity is bounded explicitly in terms of
$[k:\Q]$, $\dim A$, and $h_F(A)$.
Effective comparison results between $h_F(A)$ and the theta height are well known;
see, e.g., \cite{pazuki}.

\begin{prop}
\label{endav}
Let $A$ be a polarized abelian variety over $k$ defined by explicit equations.
Then there is an effective procedure to compute:
\begin{itemize}
\item
A finite extension $k'$ of $k$ for which we have
\[\End_{k'} (A) = \End_{\bar k} (A);\]
\item
Generators of the $\Z$-module $\End_{k'} (A)$;
\item
Generators of the group $\mathrm{NS}(A_{k'})=\mathrm{NS}(\overline{A})$
\end{itemize}
\end{prop}

\begin{proof}
First we reduce to the case when $A$ has a semi-stable model over $k$
by effective semi-stable reduction (Lemma \ref{lem:semistable}).
There is an effective bound on $[k':k]$ from Lemma 2.1 of \cite{mw1}.
The minimal $k'$ is unramified over $k$ by Theorem 1.3 of \cite{ribet}.

The main result of \cite{mw3}
(see also \cite{bost})
bounds the discriminant of the ring of
endomorphisms, which by positive-definiteness bounds the degree of
the elements in a $\Z$-basis.
So they can be found effectively.

See Lemma 2.3 of \cite{mw1}, also 5.2 of \cite{bl}, for the
(standard) identification of $\mathrm{NS}(A)\otimes\Q$
with $\End^{sym}(A)\otimes\Q$.
Pulling back the given polarization by these endomorphisms we get
generators of $\mathrm{NS}(A_{k'})\otimes\Q$.
Possibly after a further field extension, we can find representatives of
generators of $\mathrm{NS}(A_{k'})$.
\end{proof}

\begin{prop}[\cite{mw5}, Theorem 1]
\label{endmodl}
Let $A$ be an abelian variety over $k$.
Then there exists an effective $M\in\N$ such that for any $m$,
\[\End(A)\to \End_\Gamma(A_m)\]
has cokernel annihilated by $M$.
In particular, for a prime $\ell \nmid M$ the natural homomorphism
\[\End(A)/\ell\to \End_\Gamma(A_\ell)\]
is an isomorphism.
\end{prop}

\section{Effective Kuga-Satake construction}
\label{sec:kugasatake}
Let $k$ be a number field with an embedding in $\C$
and $d$ a positive integer.

\begin{defi}
\label{def:pola}
A polarization (resp.\ quasi-polarization) of degree $2d$ on a $K3$ surface $S$
over $k$ a Galois-invariant
class in $\Pic(S_{\bar k})$ which is ample (resp.\ nef) and
has self-intersection $2d$.
A primitive polarization (or quasi-polarization) is one that is not a nontrivial
multiple of another polarization (or quasi-polarization).  
\end{defi}

\begin{rema}
Suppose $S$ is given by explicit equations.
These determine a very ample line bundle $L=\mathcal{O}_S(1)$.
We can effectively determine whether the polarization $L$
is primitive, and when it is not, we can produce explicitly
a finite extension $k'$ of $k$ and a primitive polarization represented
by a line bundle $L'$ on $S_{k'}$.
If we assume, further,
that $S(k_v)\ne \emptyset$ for all places $v$ of $k$,
then a standard descent argument
(see, e.g., \S4 of \cite{kteff}) produces effectively a line bundle
on $S$ whose base change to $S_{k'}$ is isomorphic to $L'$.
\end{rema}

For the remainder of the paper we make the following assumptions.

\begin{assu}
\label{assume}
We assume there is an effective construction of
$X^\circ$ over $k$ with $X^\circ_{\C}$ isomorphic to the period space
$\Gamma\backslash \mathbb{D}$ of
primitively quasi-polarized $K3$ surfaces of degree $2d$.
Given a $K3$ surface $S$ over $k$ with explicit equations and 
supplied with an explicitly given ample
polarizing class of degree $2d$,
we assume 
we can effectively produce the corresponding point in $X^\circ$.
\end{assu}

Let $n$ be a positive integer, greater than or equal to $3$.

The Kuga-Satake construction has been treated in
\cite{deligne}, \cite{vangeemen}, \cite{kugasatake}, and
\cite{rizovcrelle}.
Here we follow the treatment in \cite{rizovcrelle}, where the
relevant level structures are described explicitly and the result
is the existence of morphisms
\[
f^{ks}_{d,a,n,\gamma}:
\mathcal{F}_{2d,n^{\mathrm{sp}}}\to \mathcal{A}_{g,d',n}
\]
of moduli spaces defined over an explicit number field.
Here, there is a standard quadratic form $Q$ on the primitive $H^2$ lattice
$P$ of the $K3$ surface, whose even Clifford algebra will be denoted
$C^+(P)$, and
$a$ is an element of the opposite algebra $C^+(P)^{\mathrm{op}}$
satisfying certain conditions.
Then $d'$ depends explicitly on $a$ and $d$, and
$\gamma$ belongs to a nonempty finite index set.
We suppose these choices are fixed.
The morphism extends to a morphism of Baily-Borel compactifications.
The compactified
source and target spaces can be constructed as projective varieties
over an explicit number field (up to
finitely many candidates) using
Propositions \ref{gammacover} and \ref{neatbb} by the
observations of Remark \ref{whenknowDmodGamma}.
Then (again up to finitely many candidates) Proposition \ref{rigidity}
produces $f^{ks}_{d,a,n,\gamma}$.

The Kuga-Satake abelian variety $A$
associated to the polarized $K3$ surface $S$ has the following
characterization (cf.\ \cite{vangeemen}).

Let $e_1$, $\ldots$, $e_{21}$ be linearly independent
vectors in $P$
diagonalizing the quadratic form so that the span of $e_1$ and $e_2$ is
negative-definite and the span of $e_3$, $\ldots$, $e_{21}$ is
positive-definite.
Let $f_1$, $f_2\in P_\R$ satisfy
$f_1+if_2\in P^{2,0}$ and $Q(f_1)=-1$.
Then $f_1$ and $f_2$ determine an element
\[J:=f_1f_2\in C^+(P)_\R.\]
The element $J$ is independent of the choice of $f_1$ and $f_2$, and
determines a complex structure on $C^+(P)_\R$.

The $\C^*$-action on $C^+(P)_\R$
\[(a+bi)\cdot x:=(a-bJ)x\]
determines a Hodge structure of weight $1$ on $C^+(P)$.
For a suitable choice of sign $\pm$,
the element $\alpha:=\pm e_1e_2\in C^+(P)$ and anti-involution
$\iota:C^+(P)\to C^+(P)$,
\[\iota(e_{i_1}\cdots e_{i_m}):=e_{e_m}\cdots e_{e_1},\qquad
(i_1<\cdots<i_m),\]
determine a polarization
\[E:C^+(P)\times C^+(P)\to \Z,\qquad
(v,w)\mapsto tr(\alpha\iota(v)w)\]
where $tr(c)$ denote the trace of the map $x\mapsto cx$.
Then the Kuga-Satake abelian variety associated with the polarized $K3$
surface $S$ is
\[A:=(C^+(P)_{\R},J)/C^+(P),\]
which is a complex torus with
polarized Hodge structure, i.e., a polarized abelian variety over $\C$.

The following properties hold.
There is an injective ring homomorphism
\begin{equation}
\label{morphismu}
u:C^+(P)\to \End(H^1(A,\Z))
\end{equation}
compactible with the weight \emph{zero} Hodge structures on source and target.
The abelian variety $A$ will be defined over a number field,
and the homomorphism of $\Z_\ell$-modules obtained
from $u$ is a homomorphism of Galois modules.

\section{Computing the Picard group of a $K3$ surface}
\label{sec:picard}
Continuing with the assumptions and notation of the previous section, we have
\[P\to\End(C^+(P))\]
sending $v$ to $y\mapsto vye_1$.
This is an injective map of Hodge structures.
From $u$ we get
an injective homomorphism of Hodge structures
\[\End(C^+(P))\to H^2(A\times A,\Z)\]

The intersection of the $(0,0)$-part of $\End(C^+(P))_\C$ with
$\End(C^+(P))$ may be effectively computed by identifying
$\End(C^+(P))_\Q$ with a direct summand of
$H^2(A\times A,\Q)$, thereby reducing the computation to the determination
of the N\'eron-Severi group of a polarized
abelian variety.

\begin{prop}
\label{effectivePicK3}
Let $S$ be a $K3$ surface as in Assumption \ref{assume}.
Then there is an effective procedure to compute $\Pic(S_{\bar k})$
by means of generators with explicit equations over a finite extension of $k$.
\end{prop}

\begin{rema}
\label{computetopology}
We are interested in the computation of
$H_*(X(\C),\Z)$, where $X$ is a smooth projective variety over $k$.
This can be done effectively, as explained in \cite{whitney}, by embedding
$X(\C)$ in a Euclidean space, subdividing the Euclidean space into cubes,
and intersecting with $X(\C)$.
When $\dim X=2$ this has been treated in \cite{krexp}.
When $X(\C)$ is isomorphic to
a quotient $\C^N/\Lambda$ for some $N$ and lattice
$\Lambda\subset\C^N$ this is known and standard.
\end{rema}

\begin{proof}[Proof of Proposition \ref{effectivePicK3}]
By the assumptions, we may choose a lift in
$\mathcal{F}_{2d,n^{\mathrm{sp}}}$ of the moduli point of $S$,
and hence obtain a finite set of candidates for the
Kuga-Satake abelian variety.
We compute $P\subset H^2(S(\C),\Z)$ and $E:C^+(P)\times C^+(P)\to \Z$ exactly
and $J\in C^+(P)_\R$ to high precision.
Evaluating theta functions to high precision, we may identify the
correct image point of the Kuga-Satake morphism, let us say $A^{ks}$ defined
over a finite extension of $k$,
and we may obtain an
analytic map $\C^{2^{19}}\to A^{ks}({\C})$ giving rise to
$A\eqto A^{ks}(\C)$, to arbitrarily high precision.
Using Proposition \ref{endav} we compute representative cycles for generators
of $NS(A^{ks}_{\C}\times A^{ks}_{\C})$.
Computation to sufficiently high precision determines their classes in
$H^2(A\times A,\Z)$.
This determines the classes of type $(0,0)$ in
$\End(C^+(P))$, and therefore in $P$.
For a choice of generators we have a degree bound, and by
a Hilbert scheme argument we obtain algebraic representatives of generators,
defined over a finite extension of $k$.
\end{proof}

\section{Proof for good primes}
\label{sec:good}
We keep the notation of the previous section and let
$\Gamma=\Gal(\bar k/k)$.

\begin{prop}
\label{prop:good}
Let $S$ be a $K3$ surface as in Assumption \ref{assume}.
Then there exists, effectively, an $\ell_0$ such that for all primes
$\ell > \ell_0$ we have $\Br(S_{\bar k})^\Gamma_\ell=0$.
\end{prop}

The rest of this section is devoted to the proof.
Proposition \ref{effectivePicK3} tells us that after suitably
extending $k$ we may suppose that $\Pic(S_{\bar k})$ is defined over $k$
and is known explicitly.
We suppose also that we have obtained the Kuga-Satake abelian variety
$A^{ks}$ with analytic map $A\eqto A^{ks}(\C)$ that can be computed to
arbitrarily high precision, and that
$\End(A^{ks}_{\bar k})$ is defined over $k$ and is known explicitly.
Then we may suppose that the subalgebra of
$\End(H_1(A,\Z))$ corresponding to endomorphisms of $A^{ks}$ has been
identified, i.e., that we have computed
\begin{equation}
\label{endosub}
\End(A)\subset \End(H_1(A(\C),\Z)).
\end{equation}

We have the exact sequence (cf.\ equation (5) of \cite{skorobogatovzarhin})
\begin{align*}
0 \to\Pic(S)/\ell^n \to
H^2&(S_{\bar k},\mu_{\ell^n})^\Gamma \to
\Br(S_{\bar k})_{\ell^n}^\Gamma \\
&{}\to H^1(k,\Pic(S_{\bar k})/\ell^n) \to
H^1(k,H^2(S_{\bar k},\mu_{\ell^n})).
\end{align*}
Let $\delta$ be the (absolute value of the) discriminant of the N\'eron-Severi group.
Then there is an exact sequence
\[0\to \Pic(S)\oplus T_S\to H^2(S(\C),\Z)\to K\to 0\]
where the cokernel $K$ is finite, of order $\delta$.
After tensoring with $\Z_\ell$ and using a comparison theorem this becomes
an exact sequence of Galois modules
\begin{equation}
\label{eqn:Kell}
0\to (\Pic(S)\otimes \Z_\ell)\oplus T_{S,\ell}\to
H^2(S_{\bar k},\Z_\ell(1))\to K_\ell\to 0
\end{equation}
for some finite $K_\ell$,
where $T_{S,\ell}$ is the submodule of $H^2(S_{\bar k},\Z_\ell(1))$
orthogonal to $\Pic(S)$; note that as abelian groups,
$T_{S,\ell}\cong T_S\otimes \Z_\ell$ and
$K_\ell\cong K\otimes \Z_\ell$.
In particular, for $\ell \nmid \delta$ and any $n$ we have
\[H^1(k,\Pic(S)/\ell^n)\hookrightarrow H^1(k,H^2(S_{\bar k},\mu_{\ell^n})),\]
and the 5-term exact sequence reduces to an isomorphism
\[(T_{S,\ell}/\ell^n)^\Gamma\eqto \Br(S_{\bar k})_{\ell^n}^\Gamma.\]

Applying transpose to
the homomorphism $u$ of \eqref{morphismu} we get a homomorphism
\[tu:C^+(P)\to \End(H_1(A,\Z)).\]

If $\Pic(S)$ has rank at least $2$, then we let $m\in P$ be an algebraic
class and construct
\[m\wedge T_S\subset \End(H_1(A(\C),\Z)).\]
By consideration of Hodge type, $\End(A)$ and $m\wedge T_S$ are
disjoint in $\End(H_1(A(\C),\Z))$.
Now, outside of finitely many $\ell$ (effectively) we have
an injective map
\[\End(A)/\ell\to \End(H_1(A(\C),\Z))/\ell\]
coming from \eqref{endosub} and a pair of injective maps
\[T_S/\ell\to (m\wedge T_S)/\ell\to \End(H_1(A(\C),\Z))/\ell,\]
such that the images in $\End(H_1(A(\C),\Z))/\ell$ are disjoint.

We use the natural isomorphism of Galois modules
\[A^{ks}_\ell\cong \Hom(H^1(A^{ks},\Z/\ell),\Z/\ell)\]
(cf.\ \cite[\S4.1]{skorobogatovzarhin}) and view
$(m\wedge T_{S,\ell})/\ell$ as a subgroup of $\End(A^{ks}_\ell)$.
So, we have an injective homomorphism of Galois modules
\[T_{S,\ell}/\ell\to \End(A^{ks}_\ell).\]

Applying Proposition \ref{endmodl} we have, away from an effectively
determined finite set of primes $\ell$, an isomorphism
\[\End(A^{ks})/\ell\eqto \End_\Gamma(A^{ks}_\ell).\]
We conclude, outside of an effectively determined finite set of primes $\ell$,
we have
\[(T_{S,\ell}/\ell)^\Gamma=0.\]

If $\Pic(S)$ has rank one, then we have $T_S=P$, and we repeat the
above argument using $\wedge^{20}T_S$ in place of
$m\wedge T_S$ and an identification of $T_S/\ell$ with
$(\wedge^{20}T_S)/\ell$ coming from $\wedge^{21}T_S\cong \Z$.

\section{Bad primes}
\label{sec:bad}
Here we refine the arguments of Section \ref{sec:good} to get an effective
bound on $\Br(S)/\Br(k)$.
We treat the primes excluded from consideration in Section \ref{sec:good}
one at a time, obtaining for each such prime $\ell$ an
effective bound on the order of the
$\ell$-primary subgroup of the image in $\Br(S_{\bar k})$ of
$\Br(S)$.
As in the previous section, we extend $k$ and assume that
$\Pic(S_{\bar k})$ is defined over $k$ and
the Kuga-Satake abelian variety together with its full
ring of geometric endomorphisms is defined over $k$.
We let $m$ be an integer such that the group
$K_\ell$ of \eqref{eqn:Kell} is $\ell^m$-torsion
and further extend $k$ so that
the group $\Br(S_{\bar k})_{\ell^m}$ is defined over $k$.
By \cite{kt} such a field extension may be produced effectively.
To obtain an effective bound on the order of the
$\ell$-primary subgroup of $\Br(S)$ in $\Br(S_{\bar k})$
it suffices
to produce an effective bound on
the order of the cokernel of
\[\Pic(S)/\ell^n\to
H^2(S_{\bar k},\mu_{\ell^n})^{\Gal(\bar k/k')}\]
that is independent of $n$.

The analysis of the previous section, in a refined form,
yields an effective bound for
$(T_{S,\ell}/\ell^n)^\Gamma$, independent of $n$.
Indeed, Proposition \ref{endmodl} yields the effective
annihilation of the cokernel of
$\End(A^{ks})\to \End_\Gamma(A^{ks}_{\ell^n})$.
In the portions of the argument where an injective homomorphism
of finitely generated abelian groups is tensored with
$\Z/\ell$, we obtain bounds independent of $n$
on the kernel of the homomorphism
tensored with $\Z/\ell^n$, rather than injective homomorphisms.
This suffices for the analysis.

Equation (1) of \cite{skorobogatovzarhin} yields an exact sequence
of Galois modules
\[0\to \Pic(S)/\ell^m\to H^2(S_{\bar k},\mu_{\ell^m})\to \Br(S_{\bar k})_{\ell^m}\to 0,\]
and therefore $\Gal(\bar k/k')$ acts trivially on
$H^2(S_{\bar k},\mu_{\ell^m})$.
By considering the sequence \eqref{eqn:Kell} tensored by
$\Z/\ell^m\Z$ it follows that $\Gal(\bar k/k')$ acts trivially on
$K_\ell$.

We consider $n\ge m$ in what follows.
Tensoring \eqref{eqn:Kell} with $\Z/\ell^n$ yields a four-term
exact sequence of Galois modules with one Tor term:
\begin{equation}
\label{eqn:fourterm}
0\to K_\ell\to
\Pic(S)/\ell^n\oplus T_{S,\ell}/\ell^n\to
H^2(S_{\bar k},\mu_{\ell^n})\to
K_\ell\to 0.
\end{equation}
Since $T_{S,\ell}/\ell^n\to H^2(S_{\bar k},\mu_{\ell^n})$ is
injective, it follows that
\begin{equation}
\label{eqn:isinjective}
K_\ell\to \Pic(S)/\ell^n
\end{equation}
is injective.

We split the exact sequence \eqref{eqn:fourterm} into two
short exact sequences
\begin{gather*}
0\to K_\ell\to
\Pic(S)/\ell^n\oplus T_{S,\ell}/\ell^n\to C\to 0,\\
0\to C\to
H^2(S_{\bar k},\mu_{\ell^n})\to
K_\ell\to 0.
\end{gather*}
This gives the long exact sequences of Galois cohomology
\begin{gather*}
K_\ell\hookrightarrow
\Pic(S)/\ell^n\oplus ( T_{S,\ell}/\ell^n )^\Gamma\to
C^\Gamma
\to H^1(\Gamma,K_\ell)
\to
H^1(\Gamma,\Pic(S)/\ell^n\oplus  T_{S,\ell}/\ell^n), \\
0\to
C^\Gamma\to
H^2(S_{\bar k},\mu_{\ell^n})^\Gamma
\to
K_\ell \to H^1(\Gamma, C).
\end{gather*}
Since \eqref{eqn:isinjective} is an injective homomorphism of
trivial Galois modules,
the first three terms of the top sequence split off as a short
exact sequence
\[
0\to K_\ell\to
\Pic(S)/\ell^n\oplus ( T_{S,\ell}/\ell^n )^\Gamma\to
C^\Gamma\to 0.
\]

We conclude by calculating that
\[
\frac{|H^2(S_{\bar k},\mu_{\ell^n})^\Gamma|}
{|\Pic(S)/\ell^n|}
\le
\frac{|K_\ell|\cdot |C^\Gamma|}
{|\Pic(S)/\ell^n|}=
|(T_{S,\ell}/\ell^n)^\Gamma|,
\]
which is bounded as explained above.

\end{document}